\documentclass[12pt]{amsart}
\usepackage{amsmath, amssymb, amsthm,verbatim,xypic,epic,eepic,longtable}

\numberwithin{equation}{section}

\newtheorem{thm}[equation]{Theorem}

\newtheorem{lemma}[equation]{Lemma}

\theoremstyle{remark}
\newtheorem*{remark}{Remark}
\newtheorem{defn}[equation]{Definition}

\newcommand{\F}{\mathbb{F}}

\renewcommand{\bar}[1]{#1\llap{$\overline{\phantom{\rm#1}}$}}
\newcommand{\bF}{\bar\F}

\begin{document}

\title{Low-degree planar monomials in characteristic two}

\author{Peter M\"uller}
\address{
  Institut f\"ur Mathematik,
  Universit\"at W\"urzburg,
  Am Hubland,
  D-97074 \\ W\"urzburg, Germany
}
\email{Peter.Mueller@mathematik.uni-wuerzburg.de}
\urladdr{www.mathematik.uni-wuerzburg.de/$\sim$mueller}

\author{Michael E. Zieve}
\address{
  Department of Mathematics,
  University of Michigan,
  530 Church Street,
  Ann Arbor, MI 48109-1043 USA
}
\email{zieve@umich.edu}
\urladdr{http://www.math.lsa.umich.edu/$\sim$zieve/}


\date{May 28, 2013}

\thanks{The second author thanks the NSF for support under grant DMS-1162181.}
 

\begin{abstract}
Planar functions over finite fields give rise to finite projective
planes and other combinatorial objects.  They exist only in odd
characteristic, but recently Zhou introduced an even characteristic
analogue which has similar applications.  In this paper we determine all planar functions on $\F_q$
of the form $c\mapsto ac^t$, where $q$ is a power of $2$, $t$ is an integer with $0< t\le q^{1/4}$, and $a\in\F_q^*$.
This settles and sharpens a conjecture of Schmidt and Zhou.
\end{abstract}

\maketitle


\section{Introduction}

Let $q=p^r$ where $p$ is prime and $r$ is a positive integer.  If $p$
is odd then a \textit{planar function} is a function
$f\colon\F_q\to\F_q$ such that, for every $b\in\F_q^*$, the function
$c\mapsto f(c+b)-f(c)$ is a bijection on $\F_q$.  Planar functions
have been used to construct finite projective planes \cite{DO},
relative difference sets \cite{GS}, error-correcting codes \cite{CDY},
and $S$-boxes in block ciphers which are optimally resistant to
differential cryptanalysis \cite{NK}.

If $p=2$ then there are no functions $f\colon\F_q\to\F_q$ satisfying
the defining property of a planar function, since $0$ and $b$ have the
same image as one another under the map $c\mapsto f(c+b)-f(c)$.
Recently Zhou \cite{Zhou} introduced a characteristic $2$ analogue of
planar functions, which have the same types of applications as do
odd-characteristic planar functions.  These will be the focus of the
present paper.  If $p=2$, we say that a function $f\colon\F_q\to\F_q$
is planar if, for every $b\in\F_q^*$, the function $c\mapsto
f(c+b)+f(c)+bc$ is a bijection on $\F_q$.  Schmidt and Zhou showed that any
function satisfying this definition can be used to produce a relative
difference set with parameters $(q,q,q,1)$, a finite projective plane,
and certain codes with unusual properties \cite{SZ,Zhou}.  In what
follows, whenever we refer to a planar function in characteristic $2$,
we mean a function satisfying Zhou's definition.

We will make progress towards a classification of planar polynomials
in characteristic $2$.  All known planar functions in characteristic
$2$ have the form $c\mapsto f(c)$ where $f(X)\in\F_q[X]$ is a
polynomial in which the degree of every term is the sum of at most two
powers of $2$ \cite{ScherrZ,SZ,Zhou}.  Our main result describes all
planar monomials of degree at most $q^{1/4}$:

\begin{thm} \label{main}
Let $t$ be a positive integer such that $t^4\le 2^r$, and let $a$ be any element of $\F_{2^r}^*$.
The function $c\mapsto ac^t$ is planar
on $\F_{2^r}$ if and only if $t$ is a power of $2$.
\end{thm}

When $t$ is odd, this result strengthens the main result of \cite{SZ},
which required $c\mapsto ac^t$ to be planar on infinitely many finite
extensions of $\F_{2^r}$.  In case $t$ is even, Schmidt and Zhou have
conjectured that if $c\mapsto ac^t$ is planar on infinitely many
finite extensions of $\F_{2^r}$ then $t$ must be a power of $2$
\cite[Conj.~10]{SZ}. Theorem~\ref{main} settles and sharpens this
conjecture, and gives a simpler proof of their result for odd $t$.

In the recent paper \cite{Zp}, the second author proved an analogue of
Theorem~\ref{main} over finite fields of odd cardinality.  However, in light of the 
subtle difference between the definitions
of planarity in odd and even characteristics, the
arguments in the odd characteristic proof are irrelevant here, and
vice-versa.  We also remark that our proof is completely different from
the proof of the weaker version of the ``odd $t$'' case of
Theorem~\ref{main} proved in \cite{SZ}, which involved a $14$-page
computation of the shapes of singularities of certain curves.

Our proof of Theorem~\ref{main} relies on a version of Weil's bound for singular plane curves.  We apply Weil's bound to
a carefully constructed auxiliary curve $C$ which by construction has few $\F_q$-rational points.
In order to apply Weil's bound, we must first show that $C$ is irreducible over the algebraic closure of $\F_q$.
We will show that, due to the specific form of the map $c\mapsto ac^t$, this irreducibility can
be deduced from a generalization of Capelli's 1898 result about irreducibility of binomials \cite{Capelli}.
We also sketch an alternate proof which, instead of Capelli's result, uses Lorenzini's theorem about the number of
irreducible translates of a bivariate polynomial \cite{Lorenzini}.


\section{The main result}

Our proof relies on a version of Weil's bound for (possibly singular) affine
plane curves.  Write $\bF_q$ for a fixed algebraic closure of $\F_q$.

\begin{defn}
A polynomial in\/ $\F_q[X,Y]$ is \emph{absolutely irreducible} if it
is irreducible in $\bF_q[X,Y]$.
\end{defn}

\begin{lemma} \label{weil}
Let $H(X,Y)\in\F_q[X,Y]$ be an absolutely irreducible polynomial. The
number of zeroes of $H(X,Y)$ in\/ $\F_q\times\F_q$ is at least
$q+1-(d-1)(d-2)\sqrt{q} - d$, where $d:=\deg(H)$.
\end{lemma}

\begin{remark}
The key ingredient in the proof of Lemma~\ref{weil} is Weil's bound on
the number of $\F_q$-rational points on a smooth projective curve over
$\F_q$ of prescribed genus \cite[p.~70, Cor.~3]{Weil}.
Lemma~\ref{weil} is deduced from Weil's bound in \cite[Cor.~2(b)]{LY}
and \cite[Cor.~2.5]{AP}.  Since this result has been the source of
some confusion, we now clarify the relevant literature.  The first
attempt to prove a version of Lemma~\ref{weil} occurred in
\cite[p.~331]{LN}, and was based on the mistaken notion that Weil's
bound for the number of points on a smooth projective curve remains
true for singular curves (a counterexample is $XYZ+Y^3+Y^2Z+Z^3=0$
over $\F_2$).  The next attempt was \cite[Thm.~4.9]{FJ}, which was
based on the mistaken notion that the number of points on a singular
projective curve is no bigger than the number of degree-one places in
its function field (a counterexample is $XY^2+XYZ+XZ^2+Y^3+Y^2Z+Z^3=0$
over $\F_2$).  The version of Lemma~\ref{weil} stated in
\cite[Rem.~8.4.18]{Z} is missing the final $d$.  Finally, we note that
one must replace the stated value of $k_d$ in \cite[Lemma~3.2]{AP} by
$(d-1)(d-2)/2$, in order to make the result follow from its proof.
The stated value of $k_d$ appears to have arisen from an incorrect
formula for the discriminant of a quadratic polynomial.
\end{remark}

\begin{proof}[Proof of Theorem~\ref{main}]
If $t$ is a power of $2$  then $c\mapsto ac^t$ is planar on $\F_q$ for every (even) $q$ and every $a\in\F_q^*$,
since $a(X+b)^t+aX^t+bX=ab^t+bX$ is a degree-one polynomial (and hence is bijective on $\F_q$) for every
$b\in\F_q^*$.  Henceforth assume that $t$ is a positive integer which is not a power of $2$, so $t\ge 3$.
Let $r$ be a positive integer such that $t^4\le 2^r$, and put $q:=2^r$.  Pick $a\in\F_q^*$, and suppose that the function $c\mapsto
ac^t$ is planar on $\F_q$. This means that $c\mapsto a(c+b)^t+ac^t+bc$
is bijective on $\F_q$ for every $b\in\F_q^*$. Upon composing on the
right with $c\mapsto bc$ and on the left with $c\mapsto b^{-t}c$, it
follows that $c\mapsto a((c+1)^t+c^t)+b^{2-t} c$ is bijective on
$\F_q$. Set $f(X,Y):=a((X+1)^t+X^t)+Y^{t-2}X\in \F_q[X,Y]$. Upon
replacing $b$ with $b^{-1}$ we see that $c\mapsto f(c,b)$ is bijective on $\F_q$
for all $b\in\F_q^*$. Set
\[
H(X,Y):=\frac{f(X,Y)+f(0,Y)}{X}=a\frac{(X+1)^t+X^t+1}{X}+Y^{t-2}.
\]
Let $N$ be the number of pairs $(c,b)$ of elements of $\F_q$ such that $H(c,b)=0$.
Note that if $H(c,b)=0$ for $c,b\in\F_q$, then $f(c,b)=f(0,b)$, so either $c=0$ or $b=0$.
Since $t$ is not a power of $2$, both $H(X,0)$ and $H(0,Y)$ are nonzero univariate polynomials of degree at most $t-2$;
therefore each of these polynomials has at most $t-2$ roots, so that $N\le 2(t-2)$.
Below we show that $H(X,Y)$ is
absolutely irreducible, so by Lemma~\ref{weil} we have $N\le q+1-(t-3)(t-4)\sqrt{q}-(t-2)$.
Since $q\ge t^4$ and $t\ge 3$, we compute
\begin{align*}
2(t-2) &\ge N \ge q+1-(t-3)(t-4)\sqrt{q}-(t-2) \\
&\ge t^4+1-(t-3)(t-4)t^2-(t-2) \\
&= 2(t-2)+ 7t^3-12t^2-3t+7\\
&> 2(t-2),
\end{align*}
a contradiction.

It remains to show that $H(X,Y)$ is absolutely irreducible. If this is
not the case, then Capelli's 1898 theorem about reducibility of binomials  (see e.g.\ \cite[Chapter~VI,
  Thm.~9.1]{lang:algebra}) yields a prime divisor $\ell$ of $t-2$ such
that $\frac{(X+1)^t+X^t+1}{X}$ is an $\ell$--th power. Write $t=2^mo$ with
$o$ odd. Recall that $t$ is not a power of $2$, so $o\ge3$. By taking
the derivative we see that $(X+1)^o+X^o+1$ has simple roots $0$ and
$1$. Thus the multiplicities of $0$ and $1$ of
$\frac{(X+1)^t+X^t+1}{X}=\frac{((X+1)^o+X^o+1)^{2^m}}{X}$ are $2^m-1$
and $2^m$, respectively. So $\ell$ divides $2^m$ and $2^m-1$ (also in
case $m=0$), a contradiction.
\end{proof}

\begin{remark}
The above proof yields the conclusion of Theorem~\ref{main} under a slightly weaker hypothesis than $t^4\le 2^r$,
namely that
\[
2^{r/2} >
t^2 - 7t + 12 + \frac{6t-14}{(t-2)\sqrt{t^2-10t+29} + t^2-7t+12}.
\]
\end{remark}

\begin{remark}
A different approach to proving Theorem~\ref{main} relies on the auxiliary polynomial $\bar H(X,Y)=((X+1)^t+X^t+(Y+1)^t+Y^t)/(X+Y)$.
As above, if $c\mapsto ac^t$ is planar then $\bar H(X,Y)+b^{t-2}/a$ cannot be absolutely irreducible for any $b\in\F_q^*$.
By a result of Lorenzini's \cite{Lorenzini}, it follows that $\bar H(X,Y)=F(G(X,Y))$ for some $F\in\bF_q[X]$ and $G\in\bF_q[X,Y]$
with $\deg(F)>1$.  Then a short argument implies that $t$ is the sum of two powers of $2$.  However, under this approach
the case that $t$ is the sum of two powers of $2$ requires more work.
\end{remark}

\end{document}